\documentclass[3p,10pt,a4paper,twoside,fleqn,sort&compress]{filomat}
\usepackage{amssymb,amsmath,latexsym}
\usepackage[varg]{pxfonts}
\usepackage{mathrsfs}
\newtheorem{theorem}{Theorem}[section]

\newtheorem{definition}[theorem]{Definition}
\newtheorem{example}[theorem]{Example}

\newtheorem{remark}[theorem]{Remark}

\usepackage{enumitem}

\newcommand{\FilVol}{xx}
\newcommand{\FilYear}{yyyy}
\newcommand{\FilPages}{zzz--zzz}
\newcommand{\FilDOI}{(will be added later)}
\newcommand{\FilReceived}{Received: dd Month yyyy; Accepted: dd Month yyyy}

\numberwithin{equation}{section}

\begin{document}

%

\title{Fixed-Point Theorems in 
b-Metric Spaces via a Novel Simulation Function}

\author[affil1]{Anuradha Gupta}
\ead{dishna2@yahoo.in}
\author[affil2]{Rahul Mansotra\corref{mycorrespondingauthor}}
\ead{mansotrarahul2@gmail.com}
\address[affil1]{Department of Mathematics
Delhi College of Arts and Commerce
Netaji nagar, New Delhi, 110023, India}
\address[affil2]{Department of Mathematics
Faculty of Mathematical sciences
University of Delhi
New Delhi, 110007, India}
\newcommand{\AuthorNames}{A. Gupta, R. Mansotra}

\newcommand{\FilMSC}{Primary 47H10; Secondary 54H25.}
\newcommand{\FilKeywords}{Fixed point, $\mathbb{A}_{\mathbb{R}}$-simulation function, $\mathfrak{J}_{_{\mathbb{A}_{\mathbb{R}}}}$-contraction, $b$-metric space.}
\newcommand{\FilCommunicated}{name of the Editor, mandatory}
\cortext[mycorrespondingauthor]{* Corresponding author: Rahul Mansotra}


\begin{abstract}
This paper introduces a new type of simulation function within the framework of
$b$-metric spaces, leading to the derivation of fixed-point results in this general setting. We explore the theoretical implications of these results and demonstrate their utility through a concrete example. 
\end{abstract}

\maketitle

\makeatletter
\renewcommand\@makefnmark%
{\mbox{\textsuperscript{\normalfont\@thefnmark)}}}
\makeatother

\section{Introduction and  Preliminaries}

Fixed point theory has long been a cornerstone in both theoretical and applied mathematics, offering deep insights into the behavior of nonlinear systems and algorithms. One of its most powerful tools is the Banach Contraction Mapping Theorem \cite{bm3}, which provides a rigorous foundation for proving the existence and uniqueness of fixed points under specific conditions. The theorem also guarantees that certain methods will converge to the fixed point, making it an essential result in the study of mathematical structures and solutions. The exploration of fixed points extends across various mathematical disciplines, including optimization, differential equations, numerical analysis, control theory, and game theory.
\par In recent decades, there has been a growing interest in extending both classical and contemporary results from metric fixed point theory to a wider range of generalized metric spaces. This transition has not only broadened the scope of fixed point theory but also introduced new challenges and opportunities for research in more abstract settings. These advancements have led to a deeper understanding of how fixed point results can be adapted and applied beyond traditional metric spaces. A recent survey by Van An et al. \cite{bm23} offers a comprehensive exploration of these developments, highlighting key results and their potential implications in various mathematical contexts. 
\par In most instances, this approach turned out to be remarkably straightforward, as the fixed point theorems developed in more general metric spaces could be easily derived from their counterparts in classical metric spaces through a systematic metrization process. This process is exemplified in several influential studies, including \cite{bm7}, \cite{bm8},  and \cite{bm10}.
\par However, there are certain generalized metric spaces, such as quasi metric spaces (often referred to as b-metric spaces within fixed point theory), where the transposition process typically leads to meaningful generalizations of fixed point theorems from traditional metric spaces. Bakhtin \cite{bm4} and Czerwik \cite{bm6} played a pivotal role in this development by extending the classical concept of metric space, introducing the more general notion of b-metric spaces, and thereby expanding the framework of fixed point theory and its related fields. \par In 2014, Jleli and Samet \cite{bm14} introduced the concept of $\vartheta$-contractions, providing an important generalization of the Banach contraction principle in the framework of Branciari distance spaces \cite{bm4}. Later, Ahmad et al. \cite{bm13} refined the conditions on the auxiliary function $\vartheta$ (say), leading to a comparable result in standard metric spaces. Alternatively, Khojasteh et al. \cite{bm17} established the concept of simulation functions with a view to consider a new class of contractions, called $\mathcal{Z}$-contractions. Such family generalized,
extended and improved several results that had been obtained in previous years. The simplicity
and usefulness of these contractions have inspirited many researchers to diversify it further
(see [ \cite{bm12},  \cite{bm15},  \cite{bm16}, \cite{bm19},  \cite{bm21}, and \cite{bm22} ]).
\par Building on the concept of simulation functions introduced by Khojasteh et al. \cite{bm17}, Cho et al. \cite{bm5} made a remarkable advancement in 2018 by introducing the $\mathcal{L}$-simulation function as a novel structure within the field. This groundbreaking contribution not only set a new standard but also sparked a wave of further research and development, highlighting the transformative influence of innovative methodologies on the evolution of simulation practices.  \\ Throughout this article, $\mathcal{X}$ denotes a nonempty set, $\mathbb{R}^{+}$ represents the set of positive real numbers, $\mathbb{N}$ stands for the set of positive integers, and $\mathbb{N}_{0}$ refers to the set of nonnegative integers.\\
This section begins with the following definition: \begin{definition}\cite{bm6} A map $\mathfrak{b} :\mathcal{X} \times \mathcal{X} \rightarrow [ 0, \infty )$ is said to be  $b$-metric on $\mathcal{X}$ if there exist $s\geq1$ such that  for all $x, y , z \in \mathcal{X},~\mathfrak{b} $ satisfies the following:
\begin{enumerate}[label=(\roman*)]
\item$\mathfrak{b}(x,y)=0$ if and only if $  x = y $;
    \item $\mathfrak{b}(x,y) = \mathfrak{b}(y,x)$;
    \item $\mathfrak{b}(x,z) \leq s[\mathfrak{b}(x,y) + \mathfrak{b}(y,z)]$.
\end{enumerate}
Then, $(\mathcal{X},\mathfrak{b},s)$ is called a $b$-metric space with coefficient $s$.
\end{definition}
\begin{definition}\cite{bm6}  Let $( \mathcal{X} , \mathfrak{b} , s)$ be a  $b$-metric space. Then:
\begin{enumerate}[label=(\roman*)]
    \item  A sequence $ ( a_{k})$ in $\mathcal{X}$ is said to be convergent if there is  $a \in \mathcal{X}$ such that  $ \lim \limits_{k\rightarrow \infty} \mathfrak{b}(a_{k},a )=0;$
    \item A sequence  $ (a_{k})$  in $\mathcal{X}$ is said to be cauchy  in $ \mathcal{X}$ if  $ \lim \limits_{k , m\rightarrow \infty} \mathfrak{b}(a_{k} , a_{m} )$  exists and is finite;
    \item $\mathcal{X}$ is said to be complete if for every Cauchy sequence $(a_{k})$ in $\mathcal{X}$ there is $a \in \mathcal{X}$ such that $\lim\limits_{k,m \rightarrow \infty } \mathfrak{b}(a_{k}  , a_{m}) =  \lim \limits_{k \rightarrow \infty }\mathfrak{b}( a_{k}, a ) = 0;$
    \item A function $\mathcal{S} : \mathcal{X} \to \mathcal{X}$ is said to be $b$-continuous if for $(a_{n}) \subseteq \mathcal{X}$, $a_{n} \to a$ in $(\mathcal{X},\mathfrak{b})$ we have $\mathcal{S}a_{n} \to \mathcal{S}a$ in $(\mathcal{X},\mathfrak{b})$.
     \end{enumerate}
     \end{definition}
    Following \cite{bm14},  $\Theta$ denotes the set of all mappings $\vartheta:  (0, \infty) \to (1, \infty)$ satisfies the following properties:\\
$(a)$ $\vartheta$ is increasing;\\
    $(b)$ for each sequence $\{ a_{n}\} \subseteq (0, \infty), \lim \limits_{n \to \infty}\vartheta(a_{n})=1 \iff \lim \limits_{n \to \infty}a_{n}=0;$\\
     $(c)$ there exist $t\in (0,1)$ and $d \in(0,\infty]$ such that $\lim \limits_{x \to 0^{+}} \dfrac{\vartheta(x)-1}{x^{t}}=d.$\\
     Further, Ahmad et al. \cite{bm13} replaced the condition $(c)$ with the following:\\
  $(d)$ $\vartheta$ is continuous.\\
The symbol  $\Theta^{*}$ denotes the collection of all   mappings satisfying conditions   $(a),(b) \text{ and } (d)$.
 \par Accordingly, authors in \cite{bm13} established the following Fixed Point Theorem:
 \begin{theorem}
     Every $\vartheta$-contraction on a complete metric space has a unique fixed point.
 \end{theorem}
 \par  Recently, Cho in \cite{bm5} introduced the $\mathcal{L}$-simulation function as follows:
  \begin{definition}A map $\mathcal{L}$ from $[1,\infty)\times[1,\infty)$ to $\mathbb{R}$ defines a $\mathcal{L}$-simulation function if   for all $a,b \in [1,\infty)$,  $\mathcal{L}$ satisfies the following properties:
	\begin{enumerate}[label=(\roman*)]
    \item $\mathcal{L}(1,1)=1;$
		\item $\mathcal{L}(a,b)<\dfrac{b}{a}$ for all $a,b>1$;
		\item if $(a_{n})$ and $(b_{n})$ are sequences in $(1,\infty)$  such that $1<\lim\limits_{n \to \infty} a_{n}=\lim \limits_{n \to \infty} b_{n}$,
          then $\limsup \limits_{n \to \infty} \mathcal{L}(a_n,b_n)<1$. 
	\end{enumerate}\end{definition}
   By $\mathscr{L}$  we denote the family of all $\mathcal{L}$-simulation functions.
    \par Hasanuzzaman et al. \cite{bm12} introduced the $\mathcal{L}$-contraction in  metric space as follows:	 
    \begin{definition}
        Let $(\mathcal{X},d )$ be a metric space. Then $\mathcal{T}: \mathcal{X}\to \mathcal{X}$ is called  $\mathcal{L}$-contraction with respect to  $\mathcal{L}$ if there exist $\mathcal{L} \in \mathscr{L} $ and $\vartheta \in \Theta^{*}$ such that $$ \mathcal{L}(\vartheta(d(\mathcal{T}x,\mathcal{T}y)), \vartheta(d(x,y)))\geq 1$$ for all $a,b \in \mathcal{X}$  with $d(\mathcal{T}x,\mathcal{T}y)>0$.
    \end{definition}
  Inspired by Cho’s work \cite{bm5} on 
$\mathcal{L}$-contractions in metric spaces and the contributions of  Gupta  and  Rohilla \cite{bm9} on simulation functions in $b$-metric spaces, this article introduces the concept of 
$\mathbb{A}_{\mathbb{R}}$-simulation functions. To underscore the importance and applicability of this concept, fixed point theorems are developed and substantiated with a comprehensive example that demonstrates its practical relevance.
\section{Fixed Point Theorems Using  $\mathfrak{J}_{_{\mathbb{A}_{\mathbb{R}}}}$-Contractions}
 This section begins by introducing the essential definitions and notations that underpin the theorems and proofs presented. \\
Let $\mathfrak{F_{c}}$ denote the class of all the  operators $\mathcal{F}_{c}\text{ from } [1,\infty)\times[1,\infty) \text{ to }\mathbb{R}$ such that for all $a,b\in [1,\infty)$,  satisfying the following properties: 
\begin{enumerate}[label=(\roman*)]
    \item $\mathcal{F}_{c}$ is continuous;
    \item $\mathcal{F}_{c}(x,y) \leq x$;
    \item $\mathcal{F}_{c}(x,y)=x$ implies that either $x=1$ or $y=1$;
    \item there exist $c \geq 1$ such that $\mathcal{F}_{c}(x,y)>c$ implies that $x>y$ and $\mathcal{F}_{c}(x,x)\leq c.$ \end{enumerate}
\begin{example}
		$\mathcal{F}_{c}(x,y)=\frac{x}{y}$. Here $c=1$.	
	\end{example}
    We define $\mathbb{A}_{\mathbb{R}}$-simulation function in the following:
	\begin{definition}A map $\mathfrak{J}$ from $[1,\infty)\times[1,\infty)$ to $\mathbb{R}$ defines a $\mathbb{A}_{\mathbb{R}}$-simulation function if there exist $s \geq 1$, $\mathcal{F}_{c}   \in \mathfrak{F_{c}}$ and $\vartheta \in \Theta^{*}$  such that  for all $x,y \in (1,\infty)$,  $\mathfrak{J}$ satisfying the following properties:
	\begin{enumerate}[label=(\roman*)]
		\item $\mathfrak{J}(x,y)<\mathcal{F}_{c}(y,x)$;
		\item if $(a_{n})$ and $(b_{n})$ are sequences in $(0,\infty)$  such that $$0<\liminf\limits_{n \to \infty} a_{n}\leq s(\limsup \limits_{n \to \infty} b_{n})\leq s^{2}(\liminf\limits_{n \to \infty} a_{n})<\infty$$
        and
     $$0<\liminf\limits_{n \to \infty} b_{n}\leq s(\limsup \limits_{n \to \infty} a_{n})\leq s^{2}(\liminf\limits_{n \to \infty} b_{n})<\infty,$$   then $\limsup \limits_{n \to \infty} \mathfrak{J}(\vartheta(a_n),\vartheta(b_n))<c$. 
	\end{enumerate}
Let $\bf{\mathscr{J}}$ denote the set of all $\mathbb{A}_{\mathbb{R}}$-simulation functions.
 \end{definition}
	
	\begin{example} Define $\mathfrak{J}:[1,\infty)\times[1,\infty)\rightarrow\mathbb{R}$,   $\mathcal{F}_{c}:[1,\infty) \times [1,\infty) \rightarrow \mathbb{R}$ and $\vartheta: (0, \infty) \to (1, \infty)$ by
		\begin{equation*}
			\mathfrak{J}(x,y)=\frac{y}{4x} \text{ , }  \mathcal{F}_{c}(y,x) = \frac{y}{x} \text{ and } \vartheta(x)=x+1.\end{equation*} 
Note that  $\mathfrak{J}(x,y)< \mathcal{F}_{c}(y,x)$, for all $x,y \in (1,\infty)$ and  $c =1$. Further, If $(a_{n})$ and $(b_{n})$ are sequences in  $(0,\infty)$  such that  $$0<\liminf\limits_{n \to \infty} a_{n}\leq 4(\limsup \limits_{n \to \infty} b_{n})\leq 16(\liminf\limits_{n \to \infty} a_{n})<\infty$$
        and
     $$0<\liminf\limits_{n \to \infty} b_{n}\leq 4(\limsup \limits_{n \to \infty} a_{n})\leq16(\liminf\limits_{n \to \infty} b_{n})<\infty,$$ then $\limsup \limits_{n \to \infty} \mathfrak{J}(\vartheta(a_{n}),\vartheta(b_{n}))=\limsup\limits_{n \to \infty}(\dfrac{b_{n}+1}{4(a_{n}+1)})=\dfrac{\limsup\limits_{n \to \infty}(b_{n}+1)}{\liminf \limits_{n \to \infty}4(a_{n}+1)}$.
     \\ Since $\limsup\limits_{n \to \infty}b_{n} \leq \liminf \limits_{n \to \infty}4 a_{n}, ~\limsup \limits_{n \to \infty} \mathfrak{J}(\vartheta(a_{n}),\vartheta(b_{n}))< 1$. Thus $\mathfrak{J}$ is a $\mathbb{A}_{\mathbb{R}}$-simulation function.
		 \end{example}
	\begin{definition} An operator $\mathcal{S}$ from $\mathcal{X}$ to $\mathcal{X}$ defines a $\mathfrak{J}_{_{\mathbb{A}_{\mathbb{R}}}}$-contraction if there exist $\mathfrak{J} \in \mathscr{J}, \vartheta \in \Theta^{*}$  such that for all $x,y \in \mathcal{X}$ with $\mathfrak{b}(\mathcal{S}x,\mathcal{S}y)>0$ implies  \begin{equation}
  \mathfrak{J}(\vartheta
(\mathfrak{b}(\mathcal{S}x,\mathcal{S}y)),\vartheta(\mathfrak{b}(x,y)))\geq c.\end{equation}
\end{definition}
\begin{theorem}
		Let $(\mathcal{X},\mathfrak{b},s)$ be a complete $b$-metric space with coefficient $s\geq1$ and $\mathcal{S}:\mathcal{X}\rightarrow \mathcal{X}$ be a given mapping. Suppose that $\mathcal{S}$ is a $\mathfrak{J}_{_{\mathbb{A}_{\mathbb{R}}}}$-contraction. Then $\mathcal{S}$ has a unique fixed point.
	\end{theorem}
	\begin{proof}
		Let $a_{0}\in \mathcal{X}$ and define $\mathcal{S}^na_0=a_n$, for all $n \in \mathbb{N}_{0}$. If $\mathfrak{b}(a_{n},a_{n+1})=0$ then $a_n=a_{n+1}=\mathcal{S}a_n$ becomes fixed point of $\mathcal{S}$. From this point onward, we can consider that $\mathfrak{b}(a_n, a_{n+1})\neq 0,$ for all $n\geq0$. Put $x=a_n \text{ and }y=a_{n+1}$ in inequality $(2.1)$ then
		\begin{align}
			c&\leq\mathfrak{J}(\vartheta(\mathfrak{b}(\mathcal{S}a_{n},\mathcal{S}a_{n+1})),\vartheta(\mathfrak{b}(a_n,a_{n+1})))\notag\\&=\mathfrak{J}(\vartheta(\mathfrak{b}(a_{n+1},a_{n+2})),\vartheta(\mathfrak{b}(a_n,a_{n+1})))\\&<\mathcal{F}_{c}(\vartheta(\mathfrak{b}(a_n,a_{n+1})),\vartheta(\mathfrak{b}(a_{n+1},a_{n+2}))).\notag
		\end{align}
		Thus, by the property of $\mathcal{F}_{c}$, we get
			$\vartheta (\mathfrak{b}(a_{n+1},a_{n+2}))<\vartheta(\mathfrak{b}(a_{n},a_{n+1}))$. Let us  suppose that $ \mathfrak{b}(a_n,a_{n+1})
  < \mathfrak{b}(a_{n+1},a_{n+2})$.  As $\vartheta$  is increasing, $ \vartheta(\mathfrak{b}(a_{n},a_{n+1})) \leq \vartheta (\mathfrak{b}(a_{n+1},a_{n+2}))$,  which is a contradiction. Thus, $\mathfrak{b}(a_{n+1},a_{n+2})\leq \mathfrak{b}(a_n,a_{n+1})$, for all $n\geq0.$ So, $(\mathfrak{b}(a_{n},a_{n+1}))$ is  a decreasing sequence of positive real numbers; hence $ \lim\limits_{n \to \infty}\mathfrak{b}(a_n,a_{n+1})=a\geq0$. We will show that $a=0. $
  Suppose $a>0$ then $0<a\leq sa \leq s^{2}a < \infty$.  Let $r_{n}=\mathfrak{b}(a_n,a_{n+1})$, then 
			$$0<\liminf\limits_{n \to \infty} r_{n+1}\leq s(\limsup \limits_{n \to \infty} r_{n})\leq s^{2}(\liminf\limits_{n \to \infty} r_{n})<\infty$$
        and
     $$0<\liminf\limits_{n \to \infty} r_{n}\leq s(\limsup \limits_{n \to \infty} r_{n+1})\leq s^{2}(\liminf\limits_{n \to \infty} r_{n})<\infty,$$
			hence by the property of $\mathfrak{J}$, we get 
			$\limsup \limits_{n \to \infty}\mathfrak{J}(\vartheta(r_{n+1}),\vartheta(r_{n}))<c.$ \\ Also, by inequality $(2.2)$,  $c \leq\limsup \limits_{n \to \infty}\mathfrak{J}(\vartheta(r_{n+1}),\vartheta(r_{n}))$, which leads to  a contradiction. Hence $\lim \limits_{n \to \infty} r_n=0.$
			Now, we aim to prove that $(a_n)$ is a Cauchy sequence. Let us suppose that $(a_{n})$ is not a Cauchy sequence in $(\mathcal{X},\mathfrak{b})$. Then, there exist $\varepsilon>0$ and subsequences $(a_{n_{i}})$ and $(a_{m_{i}})$ of  sequence $(a_n)$ such that $n_i$ is the smallest integer for which $n_i>m_i>i$ with 
            \begin{equation} \mathfrak{b}(a_{m_{i}},a_{n_{i}})\geq \varepsilon \mbox{ and }\mathfrak{b}(a_{m_i},a_{n_{i}-1})<\varepsilon. 
            \end{equation}
	Now,  $\varepsilon\leq \mathfrak{b}(a_{m_i},a_{n_i})\leq s(\mathfrak{b}(a_{m_i},a_{n_{i}-1})+\mathfrak{b}(a_{n_{i}-1},a_{n_i}))$, which implies that \begin{equation}\varepsilon\leq\liminf \limits_{i \to \infty}\mathfrak{b}(a_{m_i},a_{n_i})\leq s\varepsilon \text{ and }\varepsilon\leq\limsup \limits_{i \to \infty}\mathfrak{b}(a_{m_i},a_{n_i})\leq s\varepsilon.\end{equation}
			Note that $a_{m_i}\neq a_{n_i}$ as $\mathfrak{b}(a_{m_{i}},a_{n_{i}})\geq \varepsilon$. Substitute $x=a_{m_i-1}\text{ and } y=a_{n_i-1}$ in inequality $(2.1)$, we get
			\begin{align}c&\leq\mathfrak{J}(\vartheta(\mathfrak{b}(\mathcal{S}a_{n},\mathcal{S}a_{n+1})),\vartheta(\mathfrak{b}(a_n,a_{n+1})))\notag\\
				&=\mathfrak{J}(\vartheta(\mathfrak{b}(a_{m_i},a_{n_i})),\vartheta(\mathfrak{b}(a_{m_i-1},a_{n_i-1})))\\&<\mathcal{F}_{c}(\vartheta(\mathfrak{b}(a_{m_i-1},a_{n_{i}-1})),\vartheta(\mathfrak{b}(a_{m_i},a_{n_i})))\notag
			\end{align}
			which implies
			$\vartheta(\mathfrak{b}(a_{m_i},a_{n_i})) <\vartheta(\mathfrak{b}(a_{m_i-1},a_{n_i-1})))$. Let it be the case that $ \mathfrak{b}(a_{m_i-1},a_{n_i-1})< \mathfrak{b}(a_{m_i},a_{n_i}).$
 As $\vartheta$ is increasing, $ \vartheta(\mathfrak{b}(a_{m_i-1},a_{n_i-1}))\leq \vartheta(  \mathfrak{b}(a_{m_i},a_{n_i}))$ which leads to a contradiction. Also, $\varepsilon\leq \mathfrak{b}(a_{m_i},a_{n_i})<\mathfrak{b}(a_{m_i-1},a_{n_i-1})\leq$\\$ s(\mathfrak{b}(a_{m_{i}-1},a_{m_{i}})+\mathfrak{b}(a_{m_{i}},a_{n_{i}-1}))$. Thus, using inequality $(2.3)$ and taking limit superior and limit inferior as $i$ goes to infinity, we get
            \begin{equation}
                \varepsilon \leq \liminf \limits_{i \to \infty} \mathfrak{b}(a_{m_i-1},a_{n_i-1}) \leq s\varepsilon\text{ and } \varepsilon \leq \limsup \limits_{i \to \infty}\mathfrak{b}(a_{m_i-1},a_{n_i-1})\leq s\varepsilon.
            \end{equation}
            Using inequalities $(2.3)$, $(2.4)$ and $(2.6)$, we have
            \begin{align*}
                &0<\liminf \limits_{i \to \infty}\mathfrak{b}(a_{m_i},a_{n_i})\leq s\varepsilon \leq s(\limsup \limits_{i \to \infty} \mathfrak{b}(a_{m_i-1},a_{n_i-1}))\leq s^{2} \varepsilon \leq s^{2} (\liminf \limits_{i \to \infty}\mathfrak{b}(a_{m_i},a_{n_i})) < \infty \\ \text{and}\\
                &0<\liminf \limits_{i \to \infty} \mathfrak{b}(a_{m_i-1},a_{n_i-1})\leq s \varepsilon \leq s(\limsup \limits_{i \to \infty}\mathfrak{b}(a_{m_i},a_{n_i}))\leq s^{2}\varepsilon \leq s^{2}( \liminf \limits_{i \to \infty} \mathfrak{b}(a_{m_i-1},a_{n_i-1})) < \infty.
                \end{align*}
   Therefore, in light of the property of $\mathfrak{J}$, we have
			$$\limsup_{i \to \infty}\mathfrak{J}(\vartheta(\mathfrak{b}(a_{m_i},a_{n_i})),\vartheta(\mathfrak{b}(a_{m_i-1},a_{n_i-1})))<c.$$ Also, by inequality $(2.5)$, we get  $c\leq\limsup \limits_{i \to \infty}\mathfrak{J}(\vartheta(\mathfrak{b}(a_{m_i},a_{n_i})),\vartheta(\mathfrak{b}(a_{m_i-1},a_{n_i-1}))),$ which leads to a contradiction.
Hence, $\lim\limits_{n\to\infty}\mathfrak{b}(a_n,a_m)=0$.  \\ Since $(\mathcal{X},\mathfrak{b},s)$ is a complete $b$-metric space, \\
$$\lim\limits_{n,m \to \infty}\mathfrak{b}(a_n,a_m)=\lim\limits_{n \to \infty}\mathfrak{b}(a_n,z)=0,\text{ for some }z\in \mathcal{X}.$$\\ We will show  that $z$ is the unique fixed point of $\mathcal{S}$. Suppose $\mathfrak{b}(a_{n-1},z)\neq0$ and $\mathfrak{b}(a_n,\mathcal{S}z)\neq0$  for infinitely many $n$. Substitute $x=a_{n-1}$ and $y=z$ in inequality $(2.1)$, we get
\begin{align*}
	c&\leq\mathfrak{J}(\vartheta(\mathfrak{b}(\mathcal{S}a_{n-1},\mathcal{S}z)),\vartheta(\mathfrak{b}(a_{n-1},z)))\\&=\mathfrak{J}(\vartheta(\mathfrak{b}(a_{n},\mathcal{S}z)),\vartheta(\mathfrak{b}(a_{n-1},z)))\\&<\mathcal{F}_{c}(\vartheta(\mathfrak{b}(a_{n-1},z)),\vartheta(\mathfrak{b}(a_n,\mathcal{S}z))),
\end{align*}
hence, by the property of $\mathcal{F}_{c}$, we get
$\vartheta(\mathfrak{b}(a_n,\mathcal{S}z))<\vartheta(\mathfrak{b}(a_{n-1},z))$. Assume $ \mathfrak{b}(a_{n-1},z)<\mathfrak{b}(a_n,\mathcal{S}z)$. As $\vartheta$  is increasing, $  \vartheta(\mathfrak{b}(a_n,\mathcal{S}z))<\vartheta(\mathfrak{b}(a_{n-1},z))$, which is a contradiction. Thus,  $\mathfrak{b}(a_n,\mathcal{S}z) \leq  \mathfrak{b}(a_{n-1},z)$, which implies that $\lim \limits_{n \to \infty}\mathfrak{b}(a_n,\mathcal{S}z)=0.$
Now, $\mathfrak{b}(z,\mathcal{S}z)\leq s(\mathfrak{b}(z,a_n)+\mathfrak{b}(a_n,\mathcal{S}z))$ which on applying limit, gives $z=\mathcal{S}z$. Finally, we will prove  the uniqueness of the fixed point. Suppose $w$ be the another fixed point such that $z\neq w$. Then $\mathfrak{b}(z,w)>0$. On substituting $x=z$ and $y=w$ in inequality $(2.1)$, we have
$$c\leq\mathfrak{J}(\vartheta(\mathfrak{b}(\mathcal{S}z,\mathcal{S}w)),\vartheta(\mathfrak{b}(z,w)))<\mathcal{F}_{c}(\vartheta(\mathfrak{b}(z,w)),\vartheta(\mathfrak{b}(z,w)))\leq c,$$which is a contradiction. Hence $z$ becomes the unique fixed point.
\end{proof}
	\begin{example}
		Let $\mathcal{X}=\{1,2,3,4\}$. Define $\mathfrak{b}:\mathcal{X}\times \mathcal{X}\rightarrow\mathbb{R^{+}}$ given by\\$\mathfrak{b}(x,x)=0$ for all $x \in \mathcal{X}$,
		\\$\mathfrak{b}(1,2)=\mathfrak{b}(2,1)=3$,\\
        $\mathfrak{b}(2,3)=\mathfrak{b}(3,2)=\mathfrak{b}(1,3)=\mathfrak{b}(3,1)=1$,\\
        $\mathfrak{b}(1,4)=\mathfrak{b}(4,1)=\mathfrak{b}(2,4)=\mathfrak{b}(4,2)=\mathfrak{b}(3,4)=\mathfrak{b}(4,3)=4.$\\
  Clearly, $(\mathcal{X},\mathfrak{b},s)$ is a complete $b$-metric space with coefficient $s=\sqrt{3}.$ \\
  Also, define $\mathfrak{J}:[1,\infty)\times[1,\infty)\rightarrow\mathbb{R}$,   $\mathcal{F}_{c}:[1,\infty) \times [1,\infty) \rightarrow \mathbb{R}$ and $\vartheta: (0, \infty) \to (1, \infty)$ by
		\begin{equation*}
			\mathfrak{J}(x,y)=\frac{y}{\sqrt{3}x} \text{ , }  \mathcal{F}_{c}(y,x) = \frac{y}{x} \text{ and } \vartheta(x)=x+1.\end{equation*} 
Note that  $\mathfrak{J}(x,y)< \mathcal{F}_{c}(y,x)$, for all $x,y \in (1,\infty)$ and  $c =1$. Further, If $(a_{n})$ and $(b_{n})$ are sequences in $(0,\infty)$  such that  $$0<\liminf\limits_{n \to \infty} a_{n}\leq \sqrt{3}(\limsup \limits_{n \to \infty} b_{n})\leq 3(\liminf\limits_{n \to \infty} a_{n})<\infty$$
        and
     $$0<\liminf\limits_{n \to \infty} b_{n}\leq \sqrt{3}(\limsup \limits_{n \to \infty} a_{n})\leq3(\liminf\limits_{n \to \infty} b_{n})<\infty,$$ then $\limsup \limits_{n \to \infty} \mathfrak{J}(\vartheta(a_{n}),\vartheta(b_{n}))=\limsup\limits_{n \to \infty}(\dfrac{b_{n}+1}{\sqrt{3}(a_{n}+1)})=\dfrac{\limsup\limits_{n \to \infty}(b_{n}+1)}{\liminf \limits_{n \to \infty}\sqrt{3}(a_{n}+1)}$.
     \\ As $\limsup\limits_{n \to \infty}b_{n} \leq \liminf \limits_{n \to \infty}\sqrt{3}a_{n}.$ Hence, $\limsup \limits_{n \to \infty} \mathfrak{J}(\vartheta(a_{n}),\vartheta(b_{n}))< 1$. Thus $\mathfrak{J}$ is a $\mathbb{A}_{\mathbb{R}}$-simulation function. Moreover, define $\mathcal{S}:\mathcal{X}\rightarrow \mathcal{X}$   by \begin{equation*}
			\mathcal{S}x=\begin{cases} 3, \text{ when }  x \neq 4,\\ 1, \text{ else }.\end{cases}
		\end{equation*} We will now verify that $\mathcal{S}$ is a $\mathfrak{J}_{_{\mathbb{A}_{\mathbb{R}}}}$-contraction. Note that
        \begin{equation*}
            \mathfrak{b}(\mathcal{S}x, \mathcal{S}y)=\begin{cases}
                \mathfrak{b}(1,3)=1,& \text{if } x=4, y\neq 4,\\ \mathfrak{b}(1,1)=0,& \text{if } x=4, y=4,\\ \mathfrak{b}(3,3)=0, & \text{if }x\neq 4, y \neq 4,\\\mathfrak{b}(3,1)=1,& \text{if } x\neq 4, y = 4,
            \end{cases}
        \end{equation*}hence $\mathfrak{b}(\mathcal{S}x, \mathcal{S}y)>0$ if and only if $x=4, y\neq4,$ and  $x\neq 4, y=4$. Now, if $x=4,y\neq4, $ and $x\neq 4, y = 4,$ then $\mathfrak{b}(x,y)=4$ and $\mathfrak{b}(\mathcal{S}x, \mathcal{S}y)=1.$ Further,  for all $x, y \in \mathcal{X}$ with $\mathfrak{b}(\mathcal{S}x,\mathcal{S}y)>0,$ we have \begin{equation*}\mathfrak{J}(\vartheta
(\mathfrak{b}(\mathcal{S}x , \mathcal{S}y )),\vartheta(\mathfrak{b}(x,y)))=\dfrac{\vartheta(\mathfrak{b}(x,y))}{\sqrt{3} \vartheta
(\mathfrak{b}(\mathcal{S}x , \mathcal{S}y ))  }=\dfrac{\mathfrak{b}(x,y)+1}{\sqrt{3}(\mathfrak{b}(\mathcal{S}x , \mathcal{S}y )+1)} =\dfrac{4+1}{\sqrt{3}(1+1)}=\dfrac{5}{2\sqrt{3}}>1\end{equation*}
		Hence, $\mathcal{S}$ is a $\mathfrak{J}_{_{\mathbb{A}_{\mathbb{R}}}}$-contraction. By Theorem $2.5$, $\mathcal{S}$ has a unique fixed point $3$.
	\end{example}

\begin{theorem}
	Let $(\mathcal{X},\mathfrak{b},s)$ be a complete b-metric space with coefficient $s\geq1$ and $\mathcal{S}:\mathcal{X}\rightarrow \mathcal{X}$ be a $\mathfrak{b}$-continuous self-mapping. Suppose $\mathfrak{J}\in \bf{\mathscr{J}}$, $\vartheta\in \Theta^{*}$  and satisfies 
	\begin{equation}
		\mathfrak{J}\left(\vartheta(\mathfrak{b}(\mathcal{S}x,\mathcal{S}y)),\vartheta(\max\{\mathfrak{b}(x,y),\mathfrak{b}(x,\mathcal{S}x),\mathfrak{b}(y,\mathcal{S}y),\frac{\mathfrak{b}(\mathcal{S}x,y)+\mathfrak{b}(x,\mathcal{S}y)}{2s}\})\right)\geq c,
	\end{equation}
	for all $\mathcal{S}x\neq \mathcal{S}y,~x,y\in \mathcal{X}$. Then $\mathcal{S}$ has a unique fixed point.
\end{theorem}
\begin{proof}
	Proceeding in the similar manner as the proof of Theorem $2.5$, substitute $x=a_n, y=a_{n+1}$ in inequality $(2.7)$, we get
	\begin{align*}
&\mathfrak{J}(\vartheta(\mathfrak{b}(\mathcal{S}a_{n},\mathcal{S}a_{n+1})),\vartheta(\max\{\mathfrak{b}(a_{n},a_{n+1}),\mathfrak{b}(a_{n},\mathcal{S}a_{n}),\mathfrak{b}(a_{n+1},\mathcal{S}a_{n+1}), \frac{\mathfrak{b}(\mathcal{S}a_{n},a_{n+1})+\mathfrak{b}(a_{n},\mathcal{S}a_{n+1})}{2s}\}))=\\&\hspace{1cm}\mathfrak{J}(\vartheta(\mathfrak{b}(a_{n+1},a_{n+2})),\vartheta(\max\{\mathfrak{b}(a_{n},a_{n+1}),\mathfrak{b}(a_{n},a_{n+1}),\mathfrak{b}(a_{n+1},a_{n+2}), \frac{\mathfrak{b}(a_{n+1},a_{n+1})+\mathfrak{b}(a_{n},a_{n+2})}{2s}\}))\geq c.	
	\end{align*}
 Let $r_{n}=\mathfrak{b}(a_{n},a_{n+1})$, Then $\mathfrak{J}(\vartheta(r_{n+1}),\vartheta(\max\{r_{n}, r_{n},r_{n+1}, \dfrac{0+\mathfrak{b}(a_{n},a_{n+2})}{2s}\}))\geq c.$	
	Using the property of $\mathfrak{J}$, it  follows  that
	$\mathfrak{J}(\vartheta(r_{n+1}),\vartheta(\max\{r_{n},r_{n+1}, \dfrac{\mathfrak{b}(a_{n},a_{n+2})}{2s}\}))  <\mathcal{F}_{c}(\vartheta(\max\{r_n, r_{n+1},\dfrac{\mathfrak{b}(a_n,a_{n+2})}{2s}\}),\vartheta(r_{n+1})),$
using the property of $\mathcal{F}_{c}$, it follows that 
 $\vartheta(r_{n+1})<\vartheta(\max\{r_n,r_{n+1},\dfrac{\mathfrak{b}(a_n,a_{n+2})}{2s}\})$. Considering that $\max\{r_n,r_{n+1}, \dfrac{\mathfrak{b}(a_n,a_{n+2})}{2s}\}< r_{n+1}$. As $\vartheta$ is increasing, $\vartheta(\max\{r_n,r_{n+1},\dfrac{\mathfrak{b}(a_n,a_{n+2})}{2s}\})\leq \vartheta(r_{n+1})$, which leads to a contradiction. Hence, $r_{n+1}<\max\{r_n,r_{n+1},\dfrac{\mathfrak{b}(a_n,a_{n+2})}{2s}\}$.
Further, $r_{n+1}<\max\{r_n,r_{n+1},\dfrac{\mathfrak{b}(a_n,a_{n+2})}{2s}\} \leq \max\{ r_n, r_{n+1}, \dfrac{r_{n}+r_{n+1}}{2}\}=r_n$. Thus, $(r_{n})$ is  a decreasing sequence of positive reals; hence $ \lim\limits_{n \to \infty}r_{n}=a\geq0$. Following  the steps in  Theorem $2.5$, we get $\lim\limits_{n \to \infty}r_{n}=0.$  Now, we aim to prove that $(a_n)$ is a Cauchy sequence. Assume on contrary, that there exists an $\varepsilon>0$ such that subsequences  $(a_{n_i})$ and $(a_{m_i})$ of sequence $(a_{n})$ such that $n_i$ is the smallest integer for which
	$$n_i>m_i>i,~\mathfrak{b}(a_{m_i},a_{n_i})\geq\varepsilon~\text{and}~\mathfrak{b}(a_{m_i},a_{n_i-1})<\varepsilon.$$ Now,  $\varepsilon\leq \mathfrak{b}(a_{m_i},a_{n_i})\leq s(\mathfrak{b}(a_{m_i},a_{n_{i}-1})+\mathfrak{b}(a_{n_{i}-1},a_{n_i}))$, which implies that \begin{equation}\varepsilon\leq\liminf \limits_{i \to \infty}\mathfrak{b}(a_{m_i},a_{n_i})\leq s\varepsilon \text{ and }\varepsilon\leq\limsup \limits_{i \to \infty}\mathfrak{b}(a_{m_i},a_{n_i})\leq s\varepsilon.\end{equation}
	Note $a_{m_{i}}\neq a_{n_{i}} \text{ 
 as }\mathfrak{b}(a_{m_{i}},a_{n_{i}})\geq\varepsilon$. Substitute $x=a_{m_{i}}$ and $y=a_{n_{i}}$ in  inequality $(2.7)$, we get
 \begin{align}c&\leq\mathfrak{J}(\vartheta(\mathfrak{b}(a_{m_{i}},a_{n_{i}})),\vartheta(\max\{\mathfrak{b}(a_{m_{i}-1},a_{n_{i}-1}),r_{m_{i}},r_{n_{i}},\frac{\mathfrak{b}(a_{m_{i}},a_{n_{i}-1})+\mathfrak{b}(a_{m_{i}-1},a_{n_{i}})}{2s}\}))\\&<\mathcal{F}_{c}(\vartheta(\max\{\mathfrak{b}(a_{m_{i}-1},a_{n_{i}-1}),r_{m_{i}},r_{n_{i}},\frac{\mathfrak{b}(a_{m_{i}},a_{n_{i}-1})+\mathfrak{b}(a_{m_{i}-1},a_{n_{i}})}{2s}\}),\vartheta(\mathfrak{b}(a_{m_{i}},a_{n_{i}}))).\notag\end{align}
	By the property  $(iv)$ of $\mathcal{F}_{c}$, we have
	$$\vartheta(\mathfrak{b}(a_{m_{i}},a_{n_{i}}))<\vartheta(\max\{ \mathfrak{b}(a_{m_{i}-1},a_{n_{i}-1}),r_{m_{i}},r_{n_{i}},\frac{\mathfrak{b}(a_{m_{i}},a_{n_{i}-1})+\mathfrak{b}(a_{m_{i}-1},a_{n_{i}})}{2s}\}).$$ Suppose  $\max\{ \mathfrak{b}(a_{m_{i}-1},a_{n_{i}-1}),r_{m_{i}},r_{n_{i}},\dfrac{\mathfrak{b}(a_{m_{i}},a_{n_{i}-1})+\mathfrak{b}(a_{m_{i}-1},a_{n_{i}})}{2s}\}< \mathfrak{b}(a_{m_{i}},a_{n_{i}}).$ 
    As $\vartheta$ is increasing, \\$\vartheta(\max\{ \mathfrak{b}(a_{m_{i}-1},a_{n_{i}-1}),r_{m_{i}},r_{n_{i}},\dfrac{\mathfrak{b}(a_{m_{i}},a_{n_{i}-1})+\mathfrak{b}(a_{m_{i}-1},a_{n_{i}})}{2s}\}) \leq \vartheta(\mathfrak{b}(a_{m_{i}},a_{n_{i}}))$, which is a contradiction. Hence, $ \mathfrak{b}(a_{m_{i}},a_{n_{i}}) \leq \max\{ \mathfrak{b}(a_{m_{i}-1},a_{n_{i}-1}), r_{m_{i}},r_{n_{i}},\dfrac{\mathfrak{b}(a_{m_{i}},a_{n_{i}-1})+\mathfrak{b}(a_{m_{i}-1},a_{n_{i}})}{2s}\}.$ Now, consider the following three cases:\\
     Case $(i)$: If  $ \mathfrak{b}(a_{m_{i}},a_{n_{i}}) <r_{m_{i}}$ or $~r_{n_{i}}$ holds for infinitely many $i$ then  $\lim \limits_{i \to \infty}\mathfrak{b}(a_{m_{i}},a_{n_{i}})=0$, which is contrary to inequality $(2.8)$.\\
     Case $(ii)$: If   $\mathfrak{b}(a_{m_{i}},a_{n_{i}})< \mathfrak{b}(a_{m_{i}-1},a_{n_{i}-1})~$  holds for infinitely many $i$  then it follows from  Theorem $2.5$,  $(a_{n})$ is a Cauchy sequence.\\
     Case $(iii)$: Let $p_{i}=\dfrac{1}{2s}(\mathfrak{b}(a_{m_{i}},a_{n_{i}-1})+\mathfrak{b}(a_{m_{i}-1},a_{n_{i}}))$.  If $\mathfrak{b}(a_{m_{i}},a_{n_{i}})<p_{i}$ holds for infinitely many $i$ then $p_{i}\leq \dfrac{1}{2}(\mathfrak{b}(a_{m_{i}},a_{n_{i}})+2r_{n_{i}-1}+\mathfrak{b}(a_{m_{i}-1},a_{n_{i}-1}))$.
 Hence $\mathfrak{b}(a_{m_{i}},a_{n_{i}}) < p_{i}\leq\dfrac{1}{2}(\mathfrak{b}(a_{m_{i}},a_{n_{i}})+2r_{n_{i}-1}+\mathfrak{b}(a_{m_{i}-1},a_{n_{i}-1}))$. Also,  $\mathfrak{b}(a_{n_{i}-1},a_{m_{i}-1})\leq s(\mathfrak{b}(a_{m_{i}-1},a_{m_i})+\mathfrak{b}(a_{m_i},a_{n_{i}-1}))$ hence, $\limsup \limits_{i \to \infty}\mathfrak{b}(a_{m_{i}-1},a_{n_{i}-1})\leq s\varepsilon.$  By  inequality $(2.8)$ and $\lim \limits_{n \to \infty}r_{n}=0$, we get
 $$0 < \liminf \limits_{i \to \infty}\mathfrak{b}(a_{m_{i}},a_{n_{i}})\leq s(\limsup \limits_{i \to \infty}p_{i}) \leq s^{2}\varepsilon\leq s^{2}(\liminf \limits_{i \to \infty}\mathfrak{b}(a_{m_i},a_{n_i}))< \infty$$ and $$0<\liminf \limits_{i \to \infty}p_{i}\leq s\varepsilon \leq s(\limsup \limits_{i \to \infty}\mathfrak{b}(a_{m_i},a_{n_i}))\leq s^{2} \varepsilon\leq s^{2}(\liminf \limits_{i \to \infty}p_{i}) < \infty. $$
  Hence by the property of $\mathfrak{J}$, we get
$\limsup\limits_{i \to \infty}\mathfrak{J}(\vartheta(\mathfrak{b}(a_{m_{i}},a_{n_{i}})), \vartheta(p_{i})) <c,$ which is a contradiction to the inequality $(2.9)$. Hence $(a_{n})$ is a Cauchy sequence.\\ Since $(\mathcal{X},\mathfrak{b},s)$ is complete $b$-metric space, 
$$\lim\limits_{n,m \to \infty}\mathfrak{b}(a_n,a_m)=\lim_{n \to \infty}\mathfrak{b}(a_n,z)=0,~\text{for some}~z\in \mathcal{X}.$$ We will  prove that  $z$ becomes the unique fixed point of $\mathcal{S}$. 
    Also, $\mathfrak{b}(\mathcal{S}z,z) \leq s(\mathfrak{b}(\mathcal{S}z,\mathcal{S}a_{n-1})+\mathfrak{b}(\mathcal{S}a_{n-1},z)).$
 Using continuity of $\mathcal{S}$ and $\lim \limits_{n \to \infty}a_{n}=z$, we get $\mathfrak{b}(\mathcal{S}z,z)=0. $ Hence, $z$ becomes the fixed point of $\mathcal{S}.$
 Let $w$ be another fixed point of $\mathcal{S}$ such that $z \neq w$. On substitution of $x=z$ and $y=w$ in inequality $(2.7)$, we have 
\begin{align*} c  &\leq \mathfrak{J}\left(\vartheta(\mathfrak{b}(\mathcal{S}z,\mathcal{S}z)),\vartheta(\max\{\mathfrak{b}(z,z),\mathfrak{b}(z,\mathcal{S}z),\mathfrak{b}(x,\mathcal{S}z),\frac{\mathfrak{b}(\mathcal{S}z,z)+\mathfrak{b}(z,\mathcal{S}z)}{2s}\})\right)\\&<\mathcal{F}_{c}\left(\vartheta(\max\{\mathfrak{b}(z,z),\mathfrak{b}(z,z),\mathfrak{b}(z,z),\frac{\mathfrak{b}(z,z)+\mathfrak{b}(z,z)}{2s}\}),\vartheta(\mathfrak{b}(z,z))\right)\\&<\mathcal{F}_{c}\left(\vartheta(\max\{\mathfrak{b}(z,z),0,0,\frac{\mathfrak{b}(z,z)}{s}\},\vartheta(\mathfrak{b}(z,z))\right)\\&<\mathcal{F}_{c}\left(\vartheta(\mathfrak{b}(z,z)),\vartheta(\mathfrak{b}(z,z))\right)\\&\leq c,\end{align*}
which is a contradiction. Hence $z$ becomes the unique fixed point.
	\end{proof}	 
    \begin{example}
		Let $\mathcal{X}=\{1,2,3,4\}$. Define $\mathfrak{b}:\mathcal{X}\times \mathcal{X}\rightarrow\mathbb{R^{+}}$ given by\\$\mathfrak{b}(x,x)=0$ for all $x \in \mathcal{X}$,
		\\$\mathfrak{b}(1,2)=\mathfrak{b}(2,1)=3$,\\
        $\mathfrak{b}(2,3)=\mathfrak{b}(3,2)=\mathfrak{b}(1,3)=\mathfrak{b}(3,1)=1$,\\
        $\mathfrak{b}(1,4)=\mathfrak{b}(4,1)=15,$\\
        $\mathfrak{b}(2,4)=\mathfrak{b}(4,2)=\mathfrak{b}(3,4)=\mathfrak{b}(4,3)=4.$\\
  Clearly, $(\mathcal{X},\mathfrak{b},s)$ is a complete $b$-metric space with coefficient $s=3.$ \\
  Also, define $\mathfrak{J}:[1,\infty)\times[1,\infty)\rightarrow\mathbb{R}$,   $\mathcal{F}_{c}:[1,\infty) \times [1,\infty) \rightarrow \mathbb{R}$ and $\vartheta: (0, \infty) \to (1, \infty)$ by
		\begin{equation*}
			\mathfrak{J}(x,y)=\frac{y}{3x} \text{ , }  \mathcal{F}_{c}(y,x) = \frac{y}{x} \text{ and } \vartheta(x)=x+1.\end{equation*} 
Note that  $\mathfrak{J}(x,y)< \mathcal{F}_{c}(y,x)$, for all $x,y \in (1,\infty)$ and  $c =1$. Further, If $(a_{n})$ and $(b_{n})$ are sequences in $(0,\infty)$  such that  $$0<\liminf\limits_{n \to \infty} a_{n}\leq 3(\limsup \limits_{n \to \infty} b_{n})\leq 3^{2}(\liminf\limits_{n \to \infty} a_{n})<\infty$$
        and
     $$0<\liminf\limits_{n \to \infty} b_{n}\leq 3(\limsup \limits_{n \to \infty} a_{n})\leq 3^{2}(\liminf\limits_{n \to \infty} b_{n})<\infty,$$ then $\limsup \limits_{n \to \infty} \mathfrak{J}(\vartheta(a_{n}),\vartheta(b_{n}))=\limsup\limits_{n \to \infty}(\dfrac{b_{n}+1}{3(a_{n}+1)})=\dfrac{\limsup\limits_{n \to \infty}(b_{n}+1)}{\liminf \limits_{n \to \infty}3(a_{n}+1)}$.
     \\ As $\limsup\limits_{n \to \infty}b_{n} \leq \liminf \limits_{n \to \infty}3a_{n}.$ Hence, $\limsup \limits_{n \to \infty} \mathfrak{J}(\vartheta(a_{n}),\vartheta(b_{n}))< 1$. Thus $\mathfrak{J}$ is a $\mathbb{A}_{\mathbb{R}}$-simulation function. Moreover, define $\mathcal{S}:\mathcal{X}\rightarrow \mathcal{X}$   by \begin{equation*}
			\mathcal{S}x=\begin{cases} 3, \text{ when }  x \neq 4\\ 1, \text{ else }\end{cases}.
		\end{equation*} We will now verify that $\mathcal{S}$ is a $\mathfrak{J}_{_{\mathbb{A}_{\mathbb{R}}}}$-contraction. Note that
       \begin{equation*}
            \mathfrak{b}(\mathcal{S}x, \mathcal{S}y)=\begin{cases}
                \mathfrak{b}(1,3)=1,& \text{if } x=4, y\neq 4,\\ \mathfrak{b}(1,1)=0,& \text{if } x=4, y=4,\\ \mathfrak{b}(3,3)=0, & \text{if }x\neq 4, y \neq 4,\\\mathfrak{b}(3,1)=1,& \text{if } x\neq 4, y = 4,
            \end{cases}
        \end{equation*}hence $\mathfrak{b}(\mathcal{S}x, \mathcal{S}y)>0$ if and only if $x=4, y\neq4,$ and  $x\neq 4, y=4$. Now, consider the following cases:\\ Case 1: if  $x=4$ and $y\neq4,$ then   \begin{equation*} 1<\dfrac{15+1}{3(1+1)}=\dfrac{\mathfrak{b}(1,4)+1}{3(\mathfrak{b}(\mathcal{S}x , \mathcal{S}y )+1)}=\dfrac{\mathfrak{b}(\mathcal{S}x,x)+1}{3(\mathfrak{b}(\mathcal{S}x , \mathcal{S}y )+1)}\leq \mathfrak{J}(\vartheta
(\mathfrak{b}(\mathcal{S}x , \mathcal{S}y )),\vartheta(M_{s}(x,y))=\dfrac{M_{s}(x,y)+1}{3(\mathfrak{b}(\mathcal{S}x , \mathcal{S}y )+1)}\end{equation*}
where $M_{s}(x,y)= \max\{\mathfrak{b}(x,y),\mathfrak{b}(x,\mathcal{S}x),\mathfrak{b}(y,\mathcal{S}y),\dfrac{\mathfrak{b}(\mathcal{S}x,y)+\mathfrak{b}(x,\mathcal{S}y)}{2s}\}$.		Hence, $\mathcal{S}$ is a $\mathfrak{J}_{_{\mathbb{A}_{\mathbb{R}}}}$-contraction.\\
Case $2:$ if $x =1$, and $y=4$, then \begin{equation*} 1<\dfrac{15+1}{3(1+1)}=\dfrac{\mathfrak{b}(1,4)+1}{3(\mathfrak{b}(\mathcal{S}x , \mathcal{S}y )+1)}=\dfrac{\mathfrak{b}(x,y)+1}{3(\mathfrak{b}(\mathcal{S}x , \mathcal{S}y )+1)}\leq \mathfrak{J}(\vartheta
(\mathfrak{b}(\mathcal{S}x , \mathcal{S}y )),\vartheta(M_{s}(x,y))=\dfrac{M_{s}(x,y)+1}{3(\mathfrak{b}(\mathcal{S}x , \mathcal{S}y )+1)}\end{equation*} where $M_{s}(x,y)= \max\{\mathfrak{b}(x,y),\mathfrak{b}(x,\mathcal{S}x),\mathfrak{b}(y,\mathcal{S}y),\dfrac{\mathfrak{b}(\mathcal{S}x,y)+\mathfrak{b}(x,\mathcal{S}y)}{2s}\}$.		Hence, $\mathcal{S}$ is a $\mathfrak{J}_{_{\mathbb{A}_{\mathbb{R}}}}$-contraction.\\
Case $3:$ if $x \in \{2,3\}$, and $y=4$, then \begin{equation*} 1<\dfrac{15+1}{3(1+1)}=\dfrac{\mathfrak{b}(1,4)+1}{3(\mathfrak{b}(\mathcal{S}x , \mathcal{S}y )+1)}=\dfrac{\mathfrak{b}(\mathcal{S}y,y)+1}{3(\mathfrak{b}(\mathcal{S}x , \mathcal{S}y )+1)}\leq \mathfrak{J}(\vartheta
(\mathfrak{b}(\mathcal{S}x , \mathcal{S}y )),\vartheta(M_{s}(x,y))=\dfrac{M_{s}(x,y)+1}{3(\mathfrak{b}(\mathcal{S}x , \mathcal{S}y )+1)}\end{equation*} where $M_{s}(x,y)= \max\{\mathfrak{b}(x,y),\mathfrak{b}(x,\mathcal{S}x),\mathfrak{b}(y,\mathcal{S}y),\dfrac{\mathfrak{b}(\mathcal{S}x,y)+\mathfrak{b}(x,\mathcal{S}y)}{2s}\}$.		Hence, $\mathcal{S}$ is a $\mathfrak{J}_{_{\mathbb{A}_{\mathbb{R}}}}$-contraction.\\
Therefore, in all cases, $\mathcal{S}$ is a $\mathfrak{J}_{_{\mathbb{A}_{\mathbb{R}}}}$-contraction.
So, by Theorem $2.7$, $\mathcal{S}$ has a unique fixed point $3$.
	\end{example}
    \begin{remark}
       Only  increasing property of $\vartheta $  and property $(iv)$ of $\mathcal{F}_{c}$ is used throughout.
    \end{remark}



\end{document}